\documentclass{amsart}
\usepackage{amssymb,latexsym, amsmath, amscd}
\usepackage{tikz-cd}
\usepackage{enumerate}
\usepackage{multirow}
\usepackage{fullpage, booktabs}
\usepackage[all]{xy}
\usepackage{graphicx}
\usepackage{caption}
\makeatletter
\@namedef{subjclassname@2010}{%
\textup{2010} Mathematics Subject Classification}
\makeatother

\newtheorem{theorem}{Theorem}

\newtheorem{lemma}[theorem]{Lemma}
\newtheorem{proposition}[theorem]{Proposition}

\theoremstyle{definition}

\numberwithin{equation}{section}

\def\Z{\mathbb Z}

\def\C{\mathbb C}
\def\CP{\mathbb CP}
\def\sgn{\mathrm{sgn}}

\def\expb{\overline{\exp}}
\newcommand{\RP}{\mathbb RP}
\newcommand{\Q}{\mathbb{Q}}
\newcommand{\R}{\mathbb{R}}

\begin{document}

\title{Finite subset spaces of spheres}

\author{Jacob Mostovoy}


\maketitle

\begin{abstract} 
We find the complete rational homology for the finite subset spaces of a $d$-dimensional sphere. We also determine  the integral homology in top $d$ degrees and obtain a partial description of it in codimension $d$. 
\end{abstract}



\section{Introduction}
For a topological space $X$, the $n$th finite subset space  $\exp_{n}X$ is the quotient space of the Cartesian product $X^n$ under the equivalence relation 
$$(x_1,\ldots, x_n)\sim (y_1,\ldots, y_n)\Leftrightarrow\{x_1,\ldots, x_n\}= \{y_1,\ldots, y_n\}.$$

We have $\exp_1X=X$ and $\exp_2X$ coincides with the second symmetric product of $X$. The space $\exp_{n}X$ is naturally a subspace of $\exp_{n+1}X$. For an $m$-connected cell complex $X$, the spaces $\exp_{n}X$ are $(m+n-2)$ connected \cite{MS}; their colimit as $n\to\infty$, denoted by  $\exp X$, is, therefore, contractible whenever $X$ is connected\footnote{The contractibility of $\exp X$ is, in fact, a more basic result than the $(m+n-2)$-connectedness of  $\exp_{n}X$.}. It is the colimit $\exp X$ that has greatest importance in applications to other parts of mathematics (see \cite{BD}); however, the spaces $\exp_n X$ for finite $n$ have also been studied, with the case $X=S^1$ and $n=3$ attracting most attention \cite{Adamaszek}.

The homotopy type of $\exp_n X$ for $n>2$ is completely known only for one-dimensional complexes $X$. In particular, $\exp_3S^1$ is homeomorphic to a 3-sphere (Bott, \cite{Bott}) and Tuffley \cite{T0} shows that $\exp_n S^1$ is homotopy equivalent to $S^n$ for $n$ odd and to  $S^{n-1}$ for $n$ even. For $X$ of dimension two, Tuffley \cite{T3} computes the cohomology of $\exp_3X$ and the top two homology groups for any closed oriented surface $X$. When $X=S^2$, he shows that 
$$H_k(\exp_nS^2)=\begin{cases}\Z,\qquad\qquad\quad \qquad\text{for\ } k=2n\\ 0, \, \quad \qquad\qquad\qquad\text{for\ } k=2n-1\\ \Z\oplus \Z/(n-1)\Z\ \ \ \ \text{for\ } k=2n-2,\end{cases}$$
and that the rational homology is zero in all other dimensions. Taamallah in \cite{Taamallah} computes the integral homology of $\exp_4S^2$ and shows that 
$\exp_4 S^d$ is rationally homotopy equivalent to $S^{4d}\vee S^{3d}$ for $d$ even and $S^{2d+1}$ for $d$ odd.

In this note, we explore what kind of information can be obtained about the finite subset spaces of spheres
from known facts about the cohomology of unordered configuration spaces of points in $\R^d$.

Over the rationals we have the following:
\begin{theorem} \label{one}
The space $\exp_n S^d$ has the same rational homology as $S^{nd}\vee S^{(n-1)d}$ for $d$ even, and as $S^{[\frac{n+1}{2}](d+1)-1}$ 
for $d>1$ odd. 
\end{theorem}

As for the integer coefficients, we relate the top homology of $\exp_n S^d$ to the cohomology of the symmetric group:
\begin{theorem}\label{two}
For each $r< d$  there is an isomorphism
$$H_{nd-r}(\exp_n S^d)= H^r(S_n,\Z)$$
if $d$ is even. If $d$ is odd,
$$H_{nd-r}(\exp_n S^d)=H^r(S_n,\sgn),$$
unless $r=d-1$ {and} $n=2,3$, in which case
$$H_{nd-r}(\exp_n S^d)=H^r(S_n,\sgn)\oplus\Z.$$
Here $\sgn$ is the sign representation of $S_n$. 
\end{theorem}

The cohomology of the symmetric group appears here since it is related by the Cartan-Leray spectral sequence to the cohomology of the unordered configuration space $C_n(\R^d)$ of $n$ points on $\R^d$. In codimension $d$, this relation is less straightforward, so we relate the homology of $\exp_n S^d$ to the cohomology of $C_n(\R^d)$, with coefficients in the orientation local system if $d$ is odd:

\begin{theorem}\label{twoa}
Let  $n\geq 3$. When $d>1$ is odd, 
$$H_{nd-d}(\exp_n S^d)\subseteq H^d(C_n(\R^d),\Z_\pm)$$ is a subgroup of index at most 4. When $d$ is even,
$H_{nd-d}(\exp_n S^d)$ is an extension of $H^d(C_n(\R^d),\Z)$ by $\Z\oplus \Z/(n-1)\Z$. 
\end{theorem}

While these results do not determine $H_{nd-d}(\exp_n S^d)$ completely, for $d=2$ we recover the results of \cite{T3}. Tuffley's approach uses an explicit cell decomposition for $\exp_n S^2$ that generalizes the Fox-Neuwirth cell decomposition for the 1-point compactification of the configuration space $C_n(\R^2)$. Such cell decompositions can also be constructed for $d>2$ both for the 1-point compactifications of $C_n(\R^d)$ and for $\exp_n S^d$. However, their combinatorics are quite intricate even in the more basic case of the 1-point compactifications of $C_n(\R^d)$; see \cite{GS} for a discussion. Another kind of cell decompositions for $\exp_n S^d$ are provided by the simplicial machinery \cite{KS, MS}; they also seem to be of limited use in our case.


\section{Three types of finite subset spaces}

\subsection{Subset spaces for based polyhedra}
Let $X$ be a connected polyhedron with the basepoint $*$ and $Y=X-\{*\}$. 
Let $\exp_n(X,*)\subset \exp_n X$ be the subspace which consists of the subsets $s$ with $*\in s$. 
If we write\footnote{This notation conflicts somewhat with that of \cite{KS}.} $$\expb_n X=\exp_n X/\exp_n(X,*),$$
there is a cofibration
\begin{equation}\label{thecofibration}
\exp_n(X,*)\to\exp_n X\to\expb_n X.
\end{equation}
We have 
$$\exp_n X \subset \exp_{n+1} X,$$
$$\exp_n (X,*) \subset \exp_{n+1} (X,*)$$ and  $$\expb_n X \subset \expb_{n+1} X.$$ 
The successive quotients in the second and third of these filtrations actually coincide:
\begin{lemma}\label{quo}
For each $n$ 
$$\exp_{n+1}(X,*)/\exp_{n}(X,*)=\expb_n X / \expb_{n-1} X = C_n(Y)^+,$$
the 1-point compactification of the unordered configuration space of $n$ points on $Y$.
\end{lemma}
These filtrations give rise to homology spectral sequences which prove to be useful in the case of the homology with rational coefficients and $X=S^d$. In the case of integer coefficients, it will be sufficient for our purposes to consider the cofibration
\begin{equation}\label{unorcof}C_{n-1}(Y)^+ \to \expb_n S^d/ \expb_{n-2} S^d \to C_{n}(Y)^+,\end{equation}
which is obtained from (\ref{thecofibration}) by factoring out the subsets with fewer than $n-1$ points.

\subsection{Finite subset spaces with unlimited number of points}
Set $\exp X = \cup_n \exp_n X,$ $\exp (X,*) = \cup_n \exp_n (X,*)$ {and} 
$\expb\, X = \cup_n\, \expb_n X.$
The following is, essentially, well-known:
\begin{lemma}\label{contra}
The spaces  $\exp X$, $\exp (X,*)$ and  $\expb\, X$ are contractible.
\end{lemma}
\begin{proof}
Indeed, the spaces $\exp (X,*)$ and  $\expb\, X$ are connected topological monoids with the operation being the union of subsets. The sum in their homotopy groups is induced by the monoid operation and therefore, is idempotent. Since the only group with the idempotent group operation is trivial, all the homotopy groups of $\exp (X,*)$ and  $\expb\, X$ vanish. Since these spaces are CW-complexes (see, for instance, \cite{MS}), they are contractible. It follows that the same is true for $\exp X$.
\end{proof}

\subsection{The case of $X$ of dimension $d$}

If $X$ has dimension $d$, the spaces $\exp_n X$ and $\expb_n X$ both have dimension $nd$ while $\exp_n(X,*)$ is of dimension $nd-d$. It follows from the long exact sequence of the cofibration (\ref{thecofibration})
that $\exp_n X$ and $\expb_n X$ have the same homology in dimensions greater than $nd-d+1$.

More can be said when $X$ is a manifold (\cite{Handel, KS}). The quotient map 
$$X^{n-1}\to \exp_{n}(X,*)$$
$$(x_1,\ldots, x_{n-1})\mapsto \{x_1,\ldots, x_{n-1}, *\}$$
is a covering ramified over a subspace of codimension $d$. The monodromy group of this covering is $S_{n-1}$ that acts by interchanging the ``coordinates'' in $X^{n-1}$. If $X$ is oriented, this action on $X^{n-1}$ preserves the orientation if and only if the dimension $d$ of $X$ is even; when $d$ is odd, only even permutations preserve the orientation. As a consequence, the space $\exp_{n}(X,*)$ is a manifold away from a subcomplex of codimension $d$; it is orientable whenever $X$ is orientable and $d$ is even. It follows that for even $d$ we have $$H_{(n-1)d}(\exp_{n}(X,*))=\Z$$ and for $d>1$ odd $H_{(n-1)d}(\exp_{n}(X,*))$ vanishes. The same is true about the top-dimensional homology of  $\exp_n X$ and $\expb_n X$, so we can speak of the fundamental classes of these spaces whenever $X$ is oriented.

In fact, the condition that  $X$ is a manifold is unnecessarily strong. In order to conclude that $\exp_{n}(X,*)$ has a fundamental class for $d$ even and vanishing top homology for $d$ odd, it is sufficient to assume that $X$ is a pseudomanifold; that is, that the singularities of $X$ lie in codimension greater than one. Apart from manifolds, the class  of pseudomanifolds contains, for example, suspensions of manifolds. Any homology class in a topological space is an image of the fundamental class of a pseudomanifold that maps into this space.

Now, the exact sequence of the cofibration (\ref{thecofibration}) implies that when $X$ is a connected pseudomanifold of dimension $d$, the spaces $\exp_n X$ and $\expb_n X$ have the same homology in dimension $nd-d+1$ if $d$ is odd or if $H_{nd-d+1}(\expb_n X)$ is finite.

In turn, the homology of $\expb_n X$ in these dimensions can be expressed in terms of the homology of $C_n^+(Y)$ via Lemma~\ref{quo}. We get the following:

\begin{proposition}\label{generaltwo}
Let $X$ be a connected pseudomanifold of dimension $d$. 
$$H_{nd-r}(\exp_n X)=H_{nd-r}(C_n^+(Y))$$
whenever $r<d-1$. This also holds for $r=d-1$ whenever $d$ is odd.
\end{proposition}

With the exception of the case of $r=d-1$ with $d$ even, Theorem~\ref{two} is a consequence of this statement and the facts about the topology of the configuration spaces in $\R^d$ (Theorems~\ref{confsym} and \ref{duality} in the next section).

\subsection{A torsion class in codimension $d$}
When $X$ is an even-dimensional oriented pseudomanifold of dimension $d$ there are two distinguished classes in $H_{nd-d}(\exp_n X)$. These are the images of the fundamental classes of $\exp_{n-1} X$ and of $\exp_{n}(X,*)$; denote them by $\alpha$ and $\beta$, respectively.
\begin{proposition}\label{alphabeta}
If $X$ is a suspension of an odd-dimensional pseudomanifold, 
 $$(n-1)(\alpha - 2\beta)=0$$  in $H_{nd-d}(\exp_n X)$.
\end{proposition}
For $X=S^2$, this fact was established in \cite{T3}.
\begin{proof}
Let $$Q_{n-1}(X)\subset X\times \exp_{n-1} X$$ 
be the subspace of pairs $(x, s)$ such that $x\in s$. In other words,  $Q_{n-1}(X)$ is the space of finite subsets of $X$ of cardinality at most $n-1$, one of whose elements is marked. 

The space $Q_{n-1}(X)$ is a pseudomanifold of dimension $(n-1)d$. The projection $$\pi: Q_{n-1}(X)\to \exp_{n-1} X$$  preserves the orientation and is generically $n-1$ to one. It follows that the map of fundamental classes induced by this projection is multiplication by $n-1$ so that the image of the fundamental class of $Q_{n-1}(X)$ in $H_{nd-d}(\exp_n X)$ is $(n-1)\alpha$. 

In the same fashion, define $R_{n-1}(X)\subset  X\times \exp_{n}(X,*)$ 
as the closure of the subspace of pairs $(x, s)$ such that $x\in s$ and $x\neq *$. Again, $R_{n-1}(X)$ is a pseudomanifold of dimension $(n-1)d$; the set of points of $R_{n-1}(X)$ with $x=s$ is of dimension $(n-2)d$. The map
$$R_{n-1}(X)\to\exp_{n}(X,*)$$ that forgets the marked point is generically $(n-1)$ to one; again, the image of the fundamental class of $R_{n-1}(X)$ in $H_{nd-d}(\exp_n X)$ is $(n-1)\beta$.

Assume now that $X=\Sigma A$ and thus there is a ``vertical'' direction in $X$, along the suspension coordinate.
Define $$P_n(X)\subset \exp_2 (A\times [0,1])\times  \exp_n (A\times [0,1])$$
as the subspace of pairs $(z, s)$ such that $z\subseteq s$ and, in the case when $z$ consists of two, rather than one, points, they lie one ``above'' the other.

The space $P_n(X)$ may be thought of as a ``pseudomanifold with boundary'' in the sense that the singularities of $P_n(X)$ not only lie in codimension $>1$ but also meet the set of boundary points in codimension $>1$. The boundary of  $P_n(X)$ consists of 3 kinds of pairs: 
\begin{enumerate}
\item one of the elements in   $\exp_n (A\times [0,1])$ lies in $A\times\{0\}$;
\item one of the elements in   $\exp_n (A\times [0,1])$ lies in $A\times\{1\}$;
\item the two marked elements coincide; that is, $z\in \exp_2 (A\times [0,1])$ is a singleton.
\end{enumerate}
We denote the corresponding parts of the boundary by $U_1$, $U_2$ and $U_3$ respectively.
The map 
$$P_n(X)\to \exp_n(X)$$
that is induced by the map $A\times [0,1]\to X$ and forgets the marked points sends both $U_1$ and $U_2$ to the class $(n-1)\beta$ while $U_3$ is sent into $-(n-1)\alpha$; thus, $(n-1)(2\beta-\alpha)$ is a boundary.
\end{proof}



\section{The cohomology of the configuration spaces $C_n(\R^d)$}
We express the homology of $\exp_n X$ in terms of that of $\exp_n (X,*)$ and $\expb_n X$, and the homology of these two spaces, in turn, is computed via the configuration spaces of unordered points in $Y=X-\{*\}$. When $X$ is a sphere, we need some results about the configuration spaces of points in Euclidean spaces.

\medskip

For a space $Y$, denote by $C_n(Y)$ the configuration space of $n$ unordered distinct points on $Y$ and by $F_n(Y)$ the configuration space of $n$ ordered distinct points. The symmetric group $S_n$ acts freely on $F_n(Y)$ with the quotient $C_n(Y)$. 
When $Y$ is an odd-dimensional manifold, $C_n(Y)$ is non-orientable. Let $\Z_\pm$ be the orientation local system on $C_n(Y)$.  In what follows, we will write simply $F_n$ for $F_n(\R^d)$ and $C_n$ for  $C_n(\R^d)$.

The cohomology of the configuration spaces in Euclidean spaces is well-studied, see \cite{KConf} and references therein. We will need the following:
\begin{theorem}\label{three}
For $d$  even, $$H^{i}(C_n, \Q)=\Q$$
for $i=0,d-1$ and 0 otherwise.
For $d$ odd,
$$H^{i}(C_n, \Q_\pm)=\Q$$
for $i=[n/2](d-1)$ and 0 otherwise. Here, $\Q_\pm=\Z_\pm\otimes\Q$.
\end{theorem}
These cohomology groups can be computed as the spaces of invariants (when $d$ is even) or alternating invariants (when $d$ is odd) of the action of $S_n$ on the cohomology of the configuration spaces $F_n$. The ring $H^*(F_n, \Z)$ is generated by $n(n-1)/2$ classes $\alpha_{ij}$ of degree $d-1$; each  $\alpha_{ij}$ is the pullback of the generator of $F_2=S^{d-1}$ under the map $F_n\to F_2$ that erases all the points except  those with indices $i$ and $j$. 
The only invariants under the action of $S_n$ on $H^*(F_n, \Z)$ are in degrees 0 and $d-1$ for $d$ even; the respective spaces are 1-dimensional.  As for the twisted cohomology of $C_n$ for $d$ odd, the corresponding space of alternating invariants in $H^*(F_n, \Z)$ is described in \cite[Proposition~9]{AG}.

The integral cohomology is rather more complex; however, in a certain range, it equals the cohomology of the symmetric group $S_n$. 

\begin{theorem}\label{confsym}
Let $n>1$ and $d>2$. If $d$ is even,
$$H^r(C_n,\Z) = \begin{cases} H^r(S_n,\Z) &\text{for}\ r<d-1\\   H^{d-1}(S_n,\Z)\oplus  \Z&\text{for}\ r=d-1.\end{cases}$$
If $d$ is odd, 
$$H^r(C_n,\Z_\pm) = \begin{cases} H^r(S_n,\sgn) &\text{for}\ r<d-1\ \text{or for }  r=d-1  \ \text{and}\ n>3,  \\ H^{d-1}(S_n,\sgn)\oplus  \Z &\text{for}\ r=d-1   \ \text{and}\ n=2,3.\end{cases}$$
where $\Z_\pm$ is the orientation local system and $\sgn$ is the sign representation of $S_n$ on $\Z$. 
\end{theorem}
This statement follows from the cohomological Cartan-Leray spectral sequence for the action of $S_n$ on $F_n$ (see \cite[Ch.~VII,Theorem~7.9]{Brown} or (\cite[\S 9,Theorem~9.2]{Bredon} for the case of $d$ odd). 
The $E_2$-term of this spectral sequence has the form
$$E_2^{p,q} = H^{p}(S_n, H^q(F_n, M))$$
where $M=\Z$ considered as a trivial $S_n$-module when $d$ is even and $M=\sgn$ when $d$ is odd. 

One observes that  $F_n$ is $d-2$-connected 
and $$E_2^{0,d-1} = H^{d-1}(F_n,\Z)^{S_n}=\Z$$
for $d$ even.
Since the cohomology of $S_n$ is torsion, $E_\infty^{0,d-1}$ is of finite index in $E_2^{0,d-1}$ and, therefore, is infinite cyclic. $H^{d-1}(C_n,\Z)$ is an extension of $E_\infty^{0,d-1}=\Z$ by  $E_\infty^{d-1,0}=H^{d-1}(S_n,\Z)$; any such extension splits and this proves the statement for even $d$. 

If $d$ is odd, 
$$H^q(F_n,M)= H^q(F_n,\Z)\otimes \sgn.$$ In particular, we have
$$E^{0,d-1}_2=H^0(S_n, H^{d-1}(F_n,\Z)\otimes\sgn) = \mathrm{Hom}_{S_n}(\sgn, H^{d-1}(F_n,\Z)).$$
In addition, 
$E^{p,0}_2=H^p(S_n, \sgn)$ is a finite group for all $p$. 

We know from \cite[Proposition~9]{AG} that $H^{d-1}(F_n,\Z)$ contains the sign representation of $S_n$ only when $n=2,3$ and, in this case, the multiplicity is 1. As a consequence, $E^{0,d-1}_2$ (and, hence, $E^{0,d-1}_\infty$) is zero when $n>3$ and $\Z$ for $n=2,3$. It follows that $H^{d-1}(C_n,\Z_\pm)$ is an extension of $E^{0,d-1}_\infty$ by the finite group $E^{d-1,0}_\infty=H^{d-1}(S_n,\sgn)$ which has no other option but to split as a direct sum.

\medskip

The cohomology of the configuration spaces $F_n$ and $C_n$ is related by duality to the homology of their one-point compactifications $F_n^+$ and $C_n^+$. 
\begin{theorem}\label{duality}
For all $d$ we have
$$H_{r}(F_n^+) =   H^{nd-r}(F_n, \Z).\ $$
Also, for $d$ even 
$$H_{r}(C_n^+) =   H^{nd-r}(C_n, \Z),\ $$
and for $d$ odd
$$H_{r}(C_n^+)= H^{nd-r}(C_n, \Z_\pm).$$
\end{theorem}
Indeed, $F_n$ is the complement to a closed set $V^+$ in the compact oriented manifold $S^{nd}=((\R^{d})^n)^+$ so that
$$H_{r}(F_n^+) = H_{r}(S^{nd}/V^+)= H_{r}(S^{nd}, V^+,\Z)= H^{nd-r}(S^{nd}-V^+,\Z) =   H^{nd-r}(F_n, \Z).$$
In fact, when $d=2q$, $F_n$ may be identified with a subset of $\C^{qn}\subset (\CP^{q})^n$ on which $S_n$ acts by permutations. The quotient can be desingularized by blowups so as to obtain a compact orientable manifold and the duality applies as before. For $d$ odd, we have $F_n\subset (\RP^{d})^n$. Again, the quotient by the action of $S_n$ on $(\RP^{d})^n$ may be desingularized and the resulting manifold is non-orientable; then, one uses duality with twisted coefficients.

\medskip
Finally, let us consider the relation between the free parts of the \emph{integral} cohomology of ordered and unordered configuration spaces 
for even $d$.

The quotient map $q: F_n\to C_n$ extends to a map of 1-point compactifications $$q^+: F_n^+\to C_n^+$$
and induces a pushforward map 
$$H^{i}(F_n,\Z)=H_{nd-i}(F_n^+)\to H_{nd-i}(C_n^+)=H^{i}(C_n,\Z).$$
For $i=0$ both sides are infinite cyclic and the map is multiplication by $n!$ since, away from the basepoint at infinity, the quotient map is an $n!$-fold covering. 

The other interesting degree here is $i=d-1$. The generator $\alpha_{ij}$ of the group $H^{d-1}(F_n,\Z)$ has the following interpretation: if a ``vertical'' direction is chosen in $\R^d$, the class $\alpha_{ij}$ is the intersection number in $F_n$ with the submanifold $Y_{ij}\subset F_n$ that consists of all the configurations in which the $i$th point lies  ``directly above'' the $j$th point.
The quotient of $H^{d-1}(C_n,\Z)$ by the torsion is infinite cyclic and is generated by the intersection number in $C_n$ with the subset $Z\subset C_n$ that consists of all the configurations in which some point lies ``directly above'' another point.  The quotient map $Y_{ij}\to Z$ preserves orientation and is generically an $(n-2)!$-fold covering, so the map 
$$H^{d-1}(F_n,\Z)\to H^{d-1}(C_n,\Z)$$
sends the generator $\alpha_{ij}$ to $(n-2)!$ times a generator of $H^{d-1}(C_n,\Z)$ modulo torsion. In summary,
\begin{theorem}\label{push}
The pushforward $H^{i}(F_n,\Z)\to H^{i}(C_n,\Z)$
is a multiplication by $n!$ in degree $0$ and sends the generator $\alpha_{ij}$ to $(n-2)!$ times a generator of the free part in degree $d-1$.
\end{theorem}


\section{The homology of $\exp_n S^d$ }\label{spheres}

\subsection{The spectral sequence for the homology of $\expb_n S^d$}

The topology of $\expb_n S^d$ notably depends on whether $d$ is even or odd and, in the latter case, on whether $n$ is even or odd. Let us consider first the case when $d$ is even.

The homological spectral sequence associated with the filtration on the space $\expb_n S^d$ whose $k$th term is $\expb_k S^d$ has the $E^1$- term with
\begin{equation}E^1_{p,q}=H_{p+q}(C_{p}^+)  = H^{p(d-1)-q}(C_{p},\Z).\label{ss}\end{equation}
From Theorem~\ref{three} we have that, tensored with the rational numbers, the nonzero terms of $E^1_{p,q}$ are equal to $\Q$ for $p\leq n$ and $q=p(d-1)$ or $q=(p-1)(d-1)$.

Over $\Q$, the spectral sequence necessarily collapses at the $E^2$-term since no differential of degree 2 or higher connects non-zero terms; whatever term does not get eliminated at the first stage, remains in the limit. We claim that all the entries with the exception of a copy of $\Q$ at $E^1_{n,n(d-1)}$ cancel each other. 

Indeed, by Lemma~\ref{contra}  
 the space $\expb\, S^d= \cup_r \expb_r S^d$  is contractible. The inclusion map 
 $$\expb_n S^d \to \expb\, S^d$$ 
 respects the filtration by the subspaces $\expb_r S^d$ and the corresponding map of the $E^1$-terms the spectral sequences is the inclusion of the $E^1$-term for $\expb_n S^d$ as the first $n+1$ columns into the $E^1$-term for $\expb\, S^d$. This means that if the spectral sequence is extended infinitely to the right, beyond $p=n$, by the formula (\ref{ss}), all the nonzero terms should cancel each other at the first stage. However, the only entry that might be cancelled by a differential coming from the region $p>n$ is at $E^1_{n,n(d-1)}$ so all the other entries must be cancelled by the differentials of the spectral sequence for $\expb_n S^d$.

The entry at $E^1_{n, n(d-1)}$ is the only one remaining in the limit, so we get the following statement:

\begin{proposition}\label{uf}
For $d$ even, the space $\expb_n S^d$ is rationally equivalent to the sphere $S^{nd}$.
\end{proposition}

When $n$ is odd, we have
\begin{equation}E^1_{p,q}=H_{p+q}(C_{p}^+)  = H^{p(d-1)-q}(C_{p}, \Z_\pm).\end{equation}\label{ss2}

Over $\Q$, the nonzero terms of $E^1_{p,q}$ are equal to $\Q$ for $p\leq n$ and $q=(p-[p/2])(d-1)$.
Again, the spectral sequence has the shape that forces it to collapse at the term $E^2$. For the same reason as before, all the entries cancel each other, with the exception of the entry $E^1_{n, (n+1)(d-1)/2}$ when $n$ is even. When $n$ is odd, all the terms must cancel themselves out and we obtain the following:

\begin{figure}
\centering
\includegraphics[width=150pt]{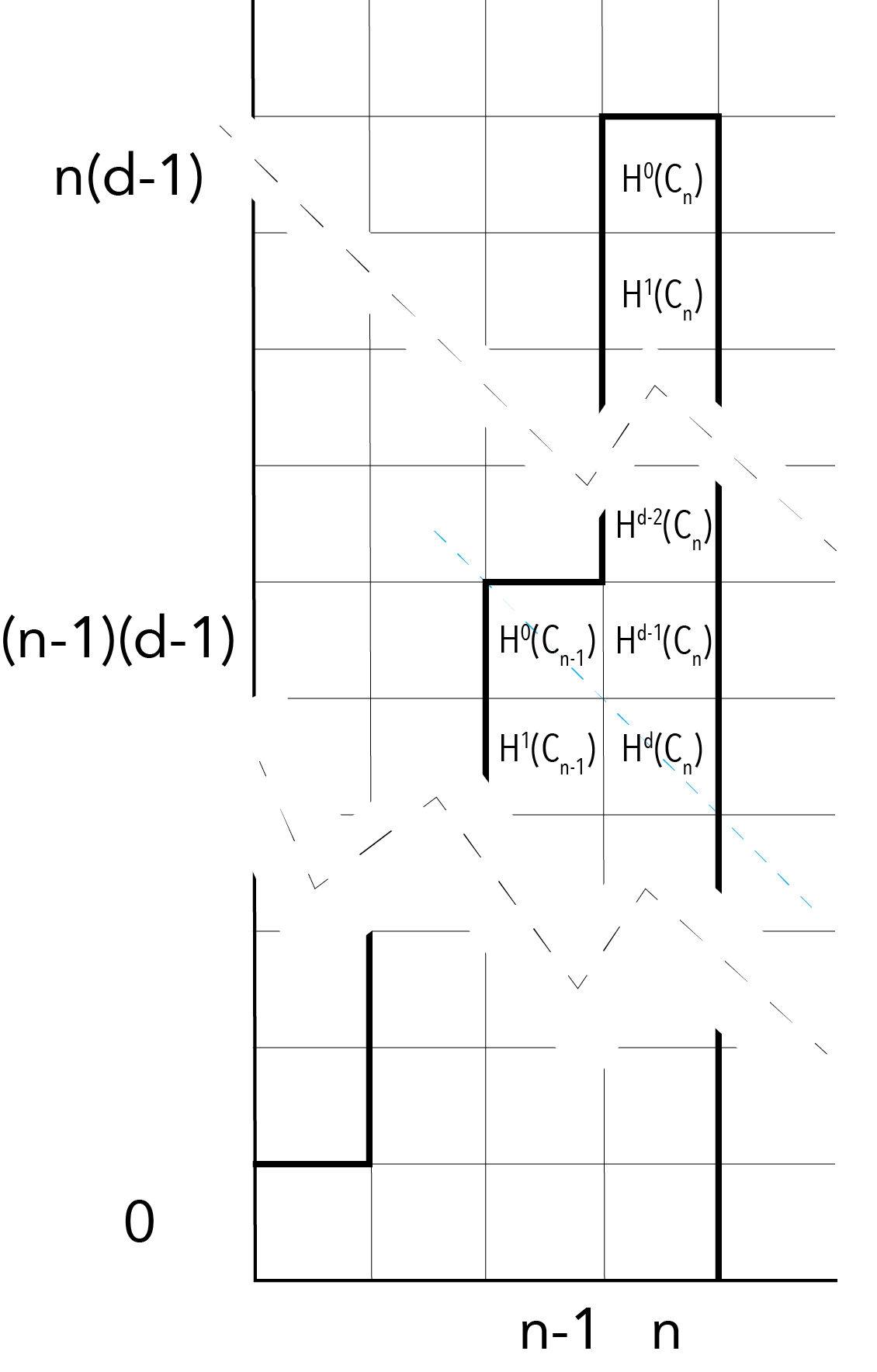}
\caption{The spectral sequence for the homology of $\expb_n S^d$}
\end{figure}

\begin{proposition}\label{uff}
For $d$ and $n$ odd, the space $\expb_n S^d$ is rationally equivalent to the sphere $S^{(n+1)(d+1)/2-1}$. For $d$ odd and $n$ even, $\expb_n S^d$ has the rational homology of a point. 
\end{proposition}

\subsection{The homology of $\expb_n S^d$ in top $d$ degrees} 

Over the integers, the spectral sequence does not necessarily collapse at the term $E^2$; for our purposes, it is sufficient to consider the cofibration (\ref{unorcof}) with $Y=R^d$:
\begin{equation}\label{unorcof2}C_{n-1}^+ \to \expb_n S^d/ \expb_{n-2} S^d \to C_{n}^+.\end{equation}

Since $C_{n-1}^+$ is a pseudomanifold of dimension $nd-d$, which is orientable whenever $d$ is even and $\expb_{n-2} S^d$ of dimension $nd-2d$, 
we have
$$H_{nd-r}(\expb_n S^d) = H_{nd-r}(C_{n}^+)$$
for all $r<d-1$. For $d$ even, the connecting map
$$H_{nd-d+1}(C_{n}^+)\to H_{nd-d}(C_{n-1}^+)=\Z$$
in the long exact homology sequence of (\ref{unorcof2}) is injective on the free part of  $H_{nd-d+1}(C_{n}^+)$: indeed, it coincides with the differential $d_1$ in the spectral sequence for $\expb_n S^d$. (We shall see later that it is, in fact, a multiplication by $n-1$). For $d$ odd,  $H_{nd-d}(C_{n-1}^+)=0$.

Theorems~\ref{confsym} and \ref{duality} give the following:
\begin{proposition}\label{pointedtwo}
We have
$$H_{nd-r}(\expb_n S^d)=
\begin{cases}H^{r}(S_n,\Z) \qquad\quad \quad   \text{for\ } d\ \text{even}, 0\leq r \leq d-1,\\
H^{r}(S_n, \sgn)\, \qquad\quad  \text{for\ } d\ \text{odd},\ 0\leq r < d-1\ \text{or}\ r=d-1\ \text{and}\ n>3,\\
H^{r}(S_n, \sgn)\oplus\Z \,\quad   \text{for\ } d\ \text{odd},\ r=d-1\ \text{and}\ n=2,3.
\end{cases}.$$
\end{proposition}

By Lemma~\ref{quo}, the successive quotients of the terms of the filtration on $\exp_{n+1} (S^d,*)$ coincide with those of $\expb_n\, S^d$. Since $\exp (S^d,*)$ is contractible, the previous arguments which concern the spaces $\expb_n\, S^d$ extend to the spaces $\exp_{n+1} (S^d,*)$ word for word. We obtain:
\begin{proposition}\label{ufff}
For $d$ even, the space $\exp_{n} (S^d,*)$ is rationally equivalent to the sphere $S^{(n-1)d}$.
For $d$ odd and $n$ even, the space $\exp_{n} (S^d,*)$ is rationally equivalent to the sphere $S^{n(d+1)/2-1}$. For $d$  and $n$ odd, $\exp_{n} (S^d,*)$  has the rational homology of a point. 
\end{proposition}


In codimension $d$ our results are somewhat less conclusive:
\begin{proposition}\label{hopefullylast}
Let  $n\geq 3$. When $d>1$ is odd, 
$$H_{nd-d}(\expb_n S^d)\subseteq H^d(C_n,\Z_\pm)$$ is a subgroup of index at most 2. When $d$ is even,
$H_{nd-d}(\expb_n S^d)$ is an extension of $H^d(C_n,\Z)$ by a cyclic group of order $n-1$. This cyclic group is generated by the image of the fundamental class of  $\expb_{n-1} S^d$. 
\end{proposition}
In particular, for $d=2$ we see that $H_{2n-2}(\expb_n S^2)=\Z/(n-1)\Z$ since the second integral cohomology of the braid group vanishes \cite{Arn2}.

\begin{proof}
Consider first the case of $d\geq 3$ odd. The homology exact sequence of the cofibration (\ref{unorcof2}) in degrees $nd-d$ and $nd-d-1$ gives
$$0\to H_{nd-d}(\expb_n S^d/\expb_{n-2} S^d)\to H_{nd-d}(C_n^+)\to H_{nd-d-1}(C_{n-1}^+)$$
that is,
$$0\to H_{nd-d}(\expb_n S^d)\to H^{d}(C_n,\Z_\pm)\to H^{1}(C_{n-1},\Z_\pm)=H^1(S_{n-1},\sgn)=\Z/2,$$
so that either $H_{nd-d}(\expb_n S^d)=H^d(C_n,\Z_\pm)$ or is a subgroup of index 2 in this group.

Let now $d$ be even. The segment 
$$H_{nd-d+1}(C_n^+)\to H_{nd-d}(C_{n-1}^+)\to H_{nd-d}(\expb_n S^d/\expb_{n-2} S^d)\to H_{nd-d}(C_n^+)\to H_{nd-d-1}(C_{n-1}^+)$$
of the homology exact sequence of (\ref{unorcof2}) is
$$H^{d-1}(S_n,\Z)\oplus\Z\xrightarrow{\partial_{nd-(d-1)}} \Z\to H_{nd-d}(\expb_n S^d)\to H^{d}(C_n,\Z)\to H^{1}(C_{n-1},\Z).$$
When $d=2$, $H^{2}(C_n,\Z)=0$ (\cite{Arn2}); when $d>2$, the last group is zero. The generator of $\Z=H_{nd-d}(C_{n-1}^+)$ is the image of the fundamental class of $\expb_{n-1} S^d$ under the quotient map that collapses $\expb_{n-2} S^d$. Proposition~\ref{hopefullylast} will be proved as soon as we show that $\partial_{nd-(d-1)}$ sends the generator of the free part of $H_{nd-d+1}(C_n^+)$ to $n-1$ times the generator of $H_{nd-d}(C_{n-1}^+)$.

Our method for computing $\partial_{nd-d+1}$ consists in covering the cofibration (\ref{unorcof2}) by an analogous cofibration which consists of the spaces of \emph{ordered} configurations.  

For $i\neq j$, let $V_{ij}\subset (\R^{d})^n$ be the codimension $d$ subspace 
$$V_{ij}=\{(x_1,\ldots, x_n)\,|\, x_i=x_j\}.$$
Then, if $$V=\bigcup_{i,j} V_{ij}$$ and
$$W=\bigcup_{\{i,j\}\neq \{i',j'\}} V_{ij}\cap V_{i'j'},$$ 
the complement of $V$ in $(\R^{d})^n$ is the configuration space $F_n$, 
while the complement of $W$ in $V$ is a disjoint union of $n(n-1)/2$ copies $F_{n-1}^{ij}$ of $F_{n-1}(\R^d)$. As before, write $S^{nd}$ for $((\R^{d})^n)^+$. 
We have $S^{nd}/V^+ = F_n^+$, $V^+/W^+ = \vee_{ij} (F_{n-1}^{ij})^+$ and there is a cofibration
\begin{equation}\label{orcof}\vee_{ij} (F_{n-1}^{ij})^+ \to S^{nd}/W^+\to  F_n^+.\end{equation}
Since $W^+$ is of codimension $2d$ in $S^{nd}$, we have $H_{nd-i}(S^{nd}/W^+)=0$ for $0<i<2d-1$. Therefore, the connecting homomorphism
$$\partial_{nd-d+1}: H_{nd-d+1}(F_n^+) \to H_{nd-d}(\vee_{ij} F_{n-1}^{ij})$$
is an isomorphism whenever $d>1$.

Now, the cofibration (\ref{orcof}) maps to the cofibration (\ref{unorcof}) by forgetting the order of the configurations. In the diagram
\[
\begin{array}{rcccccl}
H^{d-1}(F_n,\Z)&=&H_{nd-d+1}(F_n^+) & \xrightarrow{\partial_{nd-d+1}} & H_{nd-d}(\vee_{ij} (F_{n-1}^{ij})^+)&=&\Z^{n(n-1)/2} \\
&&\big\downarrow{h_1} & & \big\downarrow{h_2}&& \\
H^{d-1}(C_{n},\Z)&=&H_{nd-d+1}(C_{n}^+) & \xrightarrow{\partial_{nd-d+1}} & H_{nd-d}(C_{n-1}^+)&&\end{array}
\]
the vertical maps $h_1$ and $h_2$ are the pushforward maps described in Theorem~\ref{push}.
Each generator $\alpha_{ij}\in H^{d-1}(F_n,\Z)$, that is, the linking number with $V_{ij}$, is mapped by $h_1$ to $(n-2)!$ times the standard generator of the free part of $H^{d-1}(C_{n},\Z)$.  On the other hand, the map $h_2$ sends the image of $\alpha_{ij}$ to $(n-1)!$ times the generator, and we see that 
$\partial_{nd-d+1}$ is a multiplication by $n-1$ on the free part of $H^{d-1}(C_{n},\Z)$.

\end{proof}

\subsection{The homology of $\exp_n S^d$}

Consider the homology long exact sequence for the cofibration (\ref{thecofibration}) with $X=S^d$:
\begin{equation}\label{cofibrationspecialized}
\exp_{n} (S^d,*)\to\exp_n S^d\to \expb_n S^d.
\end{equation}

Theorem~\ref{one} follows immediately from Propositions~\ref{uf}, \ref{uff} and \ref{ufff}.
Let us prove Theorem~\ref{two}.
As noted after Proposition~\ref{generaltwo}, the only case that we need to handle is $d$ even and $r=d-1$.

When $d$ is even, $H_{nd-d}(\exp_{n} (S^d,*))=\Z$. The long exact sequence in degree $nd-d+1$ is of the form
$$
0\to H_{nd-d+1}(\exp_n S^d)\to H_{nd-d+1}(\expb_n S^d)\to\Z.
$$
The last map to $\Z$ is, actually, zero. Indeed, 
by Proposition~\ref{pointedtwo} we have $H_{nd-d+1}(\expb_n S^d)= H^{d-1}(S_n)$ which is torsion. We see that $H_{nd-r}(\exp_n S^d)$ is isomorphic to $H_{nd-r}(\expb_n S^d)$ for all $r<d$
and  Theorem~\ref{two} follows.

Finally, consider the situation in codimension $d$. 

Let $d$ be even. The homology long exact sequence of the cofibration (\ref{cofibrationspecialized})  gives
$$0 \to \Z \to H_{nd-d}(\exp_n S^d)\to H_{nd-d}(\expb_n S^d)\to 0$$
since $H_{nd-d+1}(\expb_n S^d)=H^{d-1}(S_n,\Z)$ is torsion. By Proposition~\ref{hopefullylast}, the group $H_{nd-d}(\expb_n S^d)$ is an extension of $H^d(C_n,\Z)$ by $\Z/(n-1)\Z$ and a generator of $\Z/(n-1)\Z$ may be represented by the map $$\expb_{n-1} S^d\to \expb_n S^d.$$
Let $\gamma\in H_{nd-d}(\exp_n S^d)$ be the class of $\exp_{n-1}S^d$ minus 2 times the class of $\exp_n(S^d,*)$. On one hand, $\gamma$ projects to the generator of $\Z/(n-1)\Z$ in $H_{nd-d}(\expb_n S^d)$; on the other hand, it has order $n-1$ by
Proposition~\ref{alphabeta}. It follows that $H_{nd-d}(\exp_n S^d)$ is an extension of $H^d(C_n,\Z)$
by $\Z\oplus\Z/(n-1)\Z$.

If $d$ is odd,  $H_{nd-d}(\exp_n(S^d,*))$ vanishes and $H_{nd-d-1}(\exp_n(S^d,*))=H^1(S_{n-1},\sgn)$
which is readily seen to be $\Z/2\Z$. 
Then, the homology long exact sequence of  (\ref{cofibrationspecialized}) gives
$$0 \to  H_{nd-d}(\exp_n S^d)\to H_{nd-d}(\expb_n S^d)\to \Z/2\Z$$
 and Theorem~\ref{twoa} follows from Proposition~\ref{hopefullylast}.


\end{document}